\documentclass[11pt]{amsart}

\usepackage{amsfonts, amstext, amsmath, amsthm, amscd, amssymb}
\usepackage{epsfig, graphics, psfrag}
\usepackage{color}

\let\oldmarginpar\marginpar
\renewcommand\marginpar[1]{\oldmarginpar[\raggedleft\footnotesize #1]%
{\raggedright\footnotesize #1}}

\renewcommand{\setminus}{{\smallsetminus}}

\newcommand{\ZZ}{{\mathbb{Z}}}
\newcommand{\NN}{{\mathbb{N}}}

\newcommand{\QQ}{{\mathbb{Q}}}
\newcommand{\CalD}{{\mathcal{W}}}

\newcommand{\bdy}{{\partial}}
\newcommand{\guts}{{\rm guts}}
\newcommand{\cut}{{\backslash \backslash}}

\newcommand{\vol}{{\rm vol}}

\newcommand{\negeul}{{\chi_{-}}}

\newcommand{\abs}[1]{{\left\vert #1 \right\vert}}

\newcommand{\GA}{{\mathbb{G}_A}}
\newcommand{\GB}{{\mathbb{G}_B}}
\newcommand{\G}{{\mathbb{G}}}
\newcommand{\GRA}{{\mathbb{G}'_A}}
\newcommand{\GRB}{{\mathbb{G}'_B}}
\newcommand{\A}{{\mathbb A}}

\theoremstyle{plain}
\newtheorem{theorem}{Theorem}[section]
\newtheorem{corollary}[theorem]{Corollary}
\newtheorem{lemma}[theorem]{Lemma}
\newtheorem{prop}[theorem]{Proposition}

\newtheorem*{namedtheorem}{\theoremname}
\newcommand{\theoremname}{testing}

\theoremstyle{definition}
\newtheorem{define}[theorem]{Definition}
\newtheorem{remark}[theorem]{Remark}

\begin{document}
\title[Jones polynomials, volume, and essential knot surfaces: A survey]{Jones polynomials, volume and essential knot surfaces: A survey}
\author[D. Futer]{David Futer}
\author[E. Kalfagianni]{Efstratia Kalfagianni}
\author[J. Purcell]{Jessica S. Purcell}

\address[]{Department of Mathematics, Temple University,
Philadelphia, PA 19122, USA}

\email[]{dfuter@temple.edu}

\address[]{Department of Mathematics, Michigan State University, East
Lansing, MI 48824, USA}

\email[]{kalfagia@math.msu.edu}

\address[]{ Department of Mathematics, Brigham Young University,
Provo, UT 84602, USA}

\email[]{jpurcell@math.byu.edu }
\thanks{{D.F. is supported in part by NSF grant DMS--1007221.}}

\thanks{{E.K. is supported in part by NSF grants DMS--0805942 and DMS--1105843.}}

\thanks{{J.P. is supported in part by NSF grant  DMS--1007437 and a Sloan Research Fellowship.}}

\thanks{ \today}

\begin{abstract}
This paper is a brief overview of recent results by the authors relating colored Jones polynomials to geometric topology. The proofs of these results appear  in the papers  \cite{fkp:PAMS, fkp:gutsjp}, while this survey focuses on the main ideas and examples.
\end{abstract}

\maketitle

\section*{Introduction}\label{sec:intro}

To every knot in $S^3$ there corresponds a 3--manifold, namely the
knot complement.  This 3--manifold decomposes along tori into
geometric pieces, where the most typical scenario is that all of $S^3
\setminus K$ supports a complete hyperbolic metric
\cite{thurston:bulletin}. Incompressible surfaces embedded in $S^3
\setminus K$ play a crucial role in understanding its classical
geometric and topological invariants.

The quantum knot invariants, including the Jones polynomial and its
relatives, the colored Jones polynomials, have their roots in
representation theory and physics \cite{jones, Tu}, and are well
connected to topological quantum field theory \cite{witten}.  While
the constructions of these invariants seem to be unrelated to the
geometries of 3--manifolds, in fact topological quantum field theory
predicts that the Jones polynomial knot invariants are closely related
to the hyperbolic geometry of knot complements \cite{wittengravity}.  In
particular, the \emph{volume conjecture} of R. Kashaev, H.~Murakami,
and J.~Murakami \cite{kashaev:vol-conj, murakami:vol-conj,
  murakami:volconj-survey,  gukov:vc} asserts that the volume of a hyperbolic
knot is determined by certain asymptotics of colored Jones
polynomials.  There is also growing evidence indicating direct
relations between the coefficients of the Jones and colored Jones
polynomials and the volume of hyperbolic links.  For example,
numerical computations show such relations \cite{ckp:simplest-knots},
as do theorems proved for several classes of links, including
alternating links \cite{dasbach-lin:head-tail}, closed $3$--braids
\cite{fkp:farey}, highly twisted links \cite{fkp:filling}, and certain
sums of alternating tangles \cite{fkp:conway}.

In a recent monograph \cite{fkp:gutsjp}, the authors have initiated a new approach to studying these relations, focusing on the topology of incompressible surfaces in knot
complements. The motivation behind studying surfaces is as follows.
 On the one hand, certain spanning surfaces of
knots have been shown to carry information on colored Jones
polynomials \cite{dasbach-futer...}.  On the other hand,
incompressible surfaces also shed light on volumes of manifolds
\cite{ast} and additional geometry and topology
(e.g. \cite{adams:quasi-fuchsian, menasco:incompress, miyamoto}).
With these ideas in mind, we developed  a machine that allows us to
establish relationships between colored Jones polynomials and
topological/geometric invariants.

The purpose of this paper is to give an overview of recent results,
especially those of \cite{fkp:gutsjp}, and some of their applications.
The content is an expanded version of talks given by the authors at
the conferences \emph{Topology and geometry in dimension three}, in
honor of William Jaco, at Oklahoma State University in June 2010;
\emph{Knots in Poland III} at the Banach Center in Warsaw, Poland, in
July 2010; as well as in seminars and department colloquia.  This
paper includes background and motivation, along with several examples
that did not appear in the original lectures.  Many figures in this
survey are drawn from slides for those lectures, as well as from the papers
 \cite{fkp:gutsjp, fkp:filling,
  fkp:PAMS}.
  
This paper is organized as follows. In sections \ref{sec:1D},
\ref{sec:2D}, and \ref{sec:3D}, we develop several connections between
(colored) Jones polynomials and topological objects of
the corresponding dimension.  That is, Section \ref{sec:1D} describes the
connection between these polynomial invariants and certain \emph{state
  graphs} associated to a link diagram. Section \ref{sec:2D} describes
the \emph{state surfaces} associated to these state graphs, and
explains the connection of these surfaces to the sequence of degrees
of the colored Jones polynomial. Section \ref{sec:3D} dives into the
$3$--dimensional topology of the complement of each state surface, and
contains most of our main theorems.  In Section \ref{sec:example}, we
illustrate the main theorems with a detailed example. Finally, in
Section \ref{sec:polyhedra}, we describe the polyhedral decomposition
that plays a key role in our proofs.

\section{State graphs and the Jones polynomial}\label{sec:1D}

The first objects we consider are 1--dimensional: graphs built from
the diagram of a knot or link.  We will see that these graphs
have relationships to the coefficients of the colored
Jones polynomials, and that a ribbon version of one of these graphs encodes the entire Jones polynomial.  In later sections, we will also see relationships 
between the graphs and quantities in geometric topology.

\subsection{Graphs and state graphs}

Associated to a diagram $D$ and a crossing $x$ of $D$ are two link
diagrams, each with one fewer crossing than $D$.  These are obtained
by removing the crossing $x$, and reconnecting the diagram in one of
two ways, called the \emph{$A$--resolution} and \emph{$B$--resolution}
of the crossing, shown in Figure \ref{fig:splicing}.

\begin{figure}[h]
	\centerline{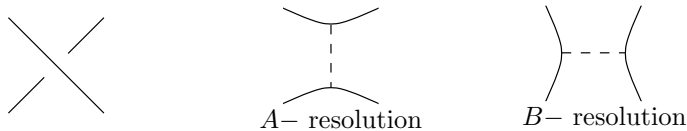}
\caption{$A$-- and $B$--resolutions of a crossing.}
\label{fig:splicing}
\end{figure}

For each crossing of $D$, we may make a choice of $A$--resolution or
$B$--resolution, and end up with a crossing--free diagram.  Such a
choice of $A$-- or $B$--resolutions is called a \emph{Kauffman state}, denoted
$\sigma$. The resulting crossing--free diagram is denoted by
$s_\sigma$.

The first graph associated with our diagram will be trivalent.  We
start with the crossing--free diagram given by a state.  The
components of this diagram are called \emph{state circles}.  For each
crossing $x$ of $D$, attach an edge from the state circle on one side
of the crossing to the other, as in the dashed lines of Figure
\ref{fig:splicing}.  Denote the resulting graph by $H_\sigma$.  Edges
of $H_\sigma$ come from state circles and crossings; there are two
trivalent vertices for each crossing.

To obtain the second graph, collapse each state circle of $H_\sigma$ to a vertex.
Denote the result by $\G_\sigma$.  The vertices of $\G_\sigma$
corespond to state circles, and the edges correspond to crossings of $D$.  The
graph $\G_\sigma$ is called the \emph{state graph} associated to
$\sigma$.

In the special case where each state circle of $\sigma$ traces a region of the diagram $D(K)$, the state graph $\G_\sigma$ is called a \emph{checkerboard graph} or \emph{Tait graph}. These checkerboard graphs record the adjacency pattern of regions of the diagram, and have been studied since the work of Tait and Listing in the 19th century. See e.g.\ \cite[Page 264]{przytycki-survey}.

\begin{figure}
\includegraphics{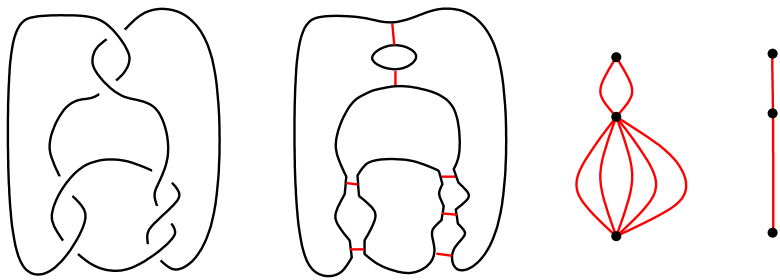}
\caption{Left to right:  A diagram, the graphs $H_A$, $\GA$, and $\GRA$. }
\label{fig:state-graph}
\end{figure}

Our primary focus from here on will be on the all--$A$ state, which
consists of choosing the $A$--resolution at each crossing, and
similarly the all--$B$ state.  Their corresponding state graphs are
denoted $\G_A$ and $\G_B$. An example of a diagram, as well as the
graphs $H_A$ and $\GA$ that result from the all--$A$ state, is shown
in Figure \ref{fig:state-graph}.

For the all--$A$ and all--$B$ states, we define graphs $\GRA$ and
$\GRB$ by removing all duplicate edges between pairs of vertices of
$\G_A$ and $\G_B$, respectively.  Again see Figure
\ref{fig:state-graph}.  

The following definition, formulated by Lickorish and Thistlethwaite
\cite{lick-thistle, thi:adequate}, captures the class of link diagrams
whose Jones polynomial invariants are especially well--behaved.  
  
\begin{define}\label{def:adequate}
A link diagram $D(K)$ is called $A$--adequate\index{$A$--adequate}
(resp. $B$--adequate\index{$B$--adequate}) if $\GA$ (resp. $\GB$) has
no 1--edge loops.  If $D(K)$ is both $A$ and $B$--adequate, then
$D(K)$ and $K$ are called \emph{adequate}\index{adequate diagram}.
\end{define}

We will devote most of our attention to $A$--adequate knots and
links. Because the mirror image of a $B$--adequate knot is
$A$--adequate, this includes the $B$--adequate knots up to reflection.
We remark that the class of $A$-- or $B$--adequate links is large.  It
includes all prime knots with up to 10 crossings, alternating links,
positive and negative closed braids, closed 3--braids, Montesinos
links, and planar cables of all the above \cite{lick-thistle,
  stoimenow, thi:adequate}.  In fact, Stoimenow has computed that
  there are only two knots of 11 crossings and a handful of 12 crossing knots that are not $A$-- or $B$--adequate.  Furthermore,
among the 253,293 prime knots with 15 crossings tabulated in Knotscape
\cite{knotscape}, at least 249,649 are either $A$--adequate or
$B$--adequate \cite{stoimenow}.  

Recall that a link diagram $D$ is called \emph{prime} if any
simple closed curve that meets the diagram transversely in two points
bounds a region of the projection plane without any crossings.  A
prime knot or link admits a prime diagram.

\subsection{The Jones polynomial from the state graph viewpoint}
Here, we recall a topological construction that allows us to recover the Jones polynomial of any knot or link from a certain $2$--dimensional embedding of $\GA$.  

A connected link diagram $D$ leads to the construction of a \emph{Turaev surface} \cite{turaevs}, as follows.  
Let $\Gamma \subset S^2$ be the planar, 4--valent graph of the link
diagram.  Thicken the (compactified) projection plane to a slab $S^2 \times
[- 1, 1]$, so that $\Gamma$ lies in $S^2 \times \{0\}$. Outside a
neighborhood of the vertices (crossings), our surface will intersect
this slab in $\Gamma \times [- 1, 1]$. In the neighborhood of
each vertex, we insert a saddle, positioned so that the boundary
circles on $S^2 \times \{1\}$ are the
components
of the $A$--resolution $s_A(D)$, and the boundary circles on $S^2
\times \{- 1\}$ are the components of $s_B(D)$. (See Figure
\ref{fig:saddle}.) Then, we cap off each circle with a disk, obtaining
an unknotted closed surface $F(D)$.

\begin{figure}[ht]
\psfrag{sa}{$s_A$}
\psfrag{sb}{$s_B$}
\psfrag{g}{$\Gamma$}
\begin{center}
\includegraphics{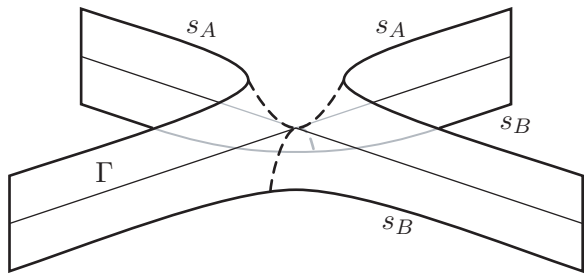}
\caption{Near each crossing of the diagram, a saddle surface interpolates 
  between circles of $s_A(D)$ and circles of $s_B(D)$. The edges of
  $\GA$ and $\GB$ can be seen as gradient lines at the saddle.}
\label{fig:saddle}
\end{center}
\end{figure}

In the special case when $D$ is an alternating diagram, each circle of
$s_A(D)$ or $s_B(D)$ follows the boundary of a region in the
projection plane. Thus, for alternating diagrams, the surface $F(D)$
is exactly the projection sphere $S^2$. For general diagrams, it is
still the case that the knot or link has an alternating projection to
$F(D)$ \cite[Lemma 4.4]{dasbach-futer...}.

The construction of the Turaev surface $F(D)$ endows it with a natural
cellulation, whose $1$--skeleton is the graph $\Gamma$ and whose
$2$--cells correspond to circles of $s_A(D)$ or $s_B(D)$, hence to
vertices of $\GA$ or $\GB$. These $2$--cells admit a natural
checkerboard coloring, in which the regions corresponding to the
vertices of $\GA$ are white and the regions corresponding to $\GB$ are
shaded. The graph $\GA$ (resp. $\GB$) can be embedded in $F(D)$ as the
adjacency graph of white (resp. shaded) regions.
Note that the \emph{faces} of $\GA$ (that is, regions in the complement of $\GA$) correspond to vertices of $\GB$, and vice versa. In other words, the graphs are dual to one another on $F(D)$.

A graph, together with an embedding into an orientable surface, is often called a \emph{ribbon graph}. Ribbon graphs and their polynomial invariants have been studied by many authors, including Bollobas and Riordan  \cite{bo-ri,
 bo-ri1}. Building on this point of view,
 Dasbach, Futer, Kalfagianni, Lin and Stoltzfus
\cite{dasbach-futer...} showed that the ribbon graph embedding of $\GA$ into the Turaev surface $F(D)$ carries at least as much information as the Jones polynomial $J_K(t)$.
To state the relevant result from \cite{dasbach-futer...},
we need the following definition.

\begin{define}\label{defi:spasubgraph}
  A \emph{spanning} subgraph of $\GA$ is a subgraph that contains all
  the vertices of $\GA$.  Given a spanning subgraph $\G$ of $\GA$ we
  will use $v(\G)$, $e(\G)$ and $f(\G)$ to denote the number of
  vertices, edges and faces of $\G$ respectively. 
\end{define}

\begin{theorem}[\cite{dasbach-futer...}]\label{JPgraph} 
Let $D$ be a connected link diagram.  Then the Kauffman bracket $\langle D \rangle\in
\ZZ[A, A^{-1}]$ can be expressed as
$$\langle D \rangle = \sum_{\G\subset 
  \GA}^{\phantom{a}}\ A^{e(\GA)-2e(\G)} (-A^2-A^{-2})^{f(\G)-1},$$
where $\G$ ranges over all the spanning subgraphs of $\GA$.
\end{theorem}

Recall that given a diagram $D$, the Jones polynomial $J_K(t)$ is
obtained from the Kauffman bracket as follows. Multiply $\langle D
\rangle$ by with $(-A)^{-3w(D)}$, where $w(D)$ is the \emph{writhe} of
$D$, and then substitute $A = t^{-1/4}$.

\begin{remark} Theorem \ref{JPgraph}  leads to formulae for the coefficients 
of the Jones polynomial of a link in terms of topological quantities
of the graph $\GA$ corresponding to any diagram of the link
\cite{dasbach-futer..., dfkls:determinant}. These formulae become
particularly effective if $\GA$ corresponds to an $A$--adequate
diagram; in particular, Theorem \ref{thm:stabilized} below can be
recovered from these formulae.
\end{remark}

The polynomial $J_K(t)$ fits within a family of knot polynomials  known as the \emph {colored Jones polynomials}. 
A convenient way to
express this family is in terms of \emph{Chebyshev polynomials}. For $n \geq 0$,
the polynomial $S_n(x)$ is defined recursively as follows:
\begin{equation}\label{eq:cheb-recursive}
S_{n+1} = x S_n - S_{n-1}, \qquad S_1(x) = x, \qquad S_0(x) = 1.
\end{equation}

Let $D$ be a diagram of a link $K$. For an integer $m > 0$, let $D^m$
denote the diagram obtained from $D$ by taking $m$ parallel copies of
$K$.  This is the $m$--cable of $D$ using the blackboard framing; if
$m=1$ then $D^1=D$.  Let $\langle D^m \rangle$ denote the Kauffman
bracket of $D^m$ and let  $w=w(D)$ denote the writhe of
$D$. Then we may define the function
\begin{equation*}\label{eq:unreduced}
  G(n+1, A):= \left((-1)^n A^{n^2+2n} \right)^{-w} (-1)^{n-1}
  \left(\frac{A^4 - A^{-4}}{A^{2n} - A^{-2n}} \right) \langle S_n( D)
  \rangle,
\end{equation*}
where $S_n(D)$ is a linear combination of blackboard cablings of $D$,
obtained via equation (\ref{eq:cheb-recursive}), and the notation
$\langle S_n(D) \rangle$ means extend the Kaufmann bracket linearly.
That is, for diagrams $D_1$ and $D_2$ and scalars $a_1$ and $b_1$,
$\langle a_1 D_1 + a_2 D_2 \rangle = a_1 \langle D_1\rangle +
a_2\langle D_2\rangle$. Finally, the reduced $n$-th colored Jones polynomial of $K$, denoted
$$J^n_K(t)= \alpha_n t^{j(n)}+ \beta_n t^{j(n)-1}+ \ldots + \beta'_n
t^{j'(n)+1}+ \alpha'_n t^{j'(n)},$$
is obtained from $G(n, A)$ by substituting
$t:=A^{-4}$.  

For a given a diagram $D$ of $K$, there is
a lower bound for $j'(n)$ in terms of data about the state graph
$\GA$, and this bound is sharp when $D$ is $A$--adequate.  Similarly,
there is an upper bound on $j(n)$ in terms of $\GB$ that is realized
when $D$ is $B$--adequate \cite{lickorish:book}. See Theorem \ref{thm:slopes} for a related statement.
In addition, Dasbach and Lin showed that for $A$-- and $B$--adequate diagrams, the extreme coefficients of $J^n_K(t)$ have a particularly nice form.

\begin{theorem}[\cite{dasbach-lin:head-tail}] \label{thm:stabilized}
If $D$ is an $A$--adequate diagram, then $\alpha'_n$ and $\beta'_n$ are independent of $n>1$. In particular, $\abs{\alpha'_n} = 1$ and $\abs{\beta'_n} = 1 - \chi(\GRA)$, where $\GRA$
  is the reduced graph.
  
Similarly, if $D$ is $B$--adequate, then $\abs{\alpha_n} = 1$ and $\abs{\beta_n} = 1 - \chi(\GRB)$.
\end{theorem}

 Now the following definition makes sense in the light of Theorem
\ref{thm:stabilized}.

\begin{define}\label{defi:stable}
For an $A$--adequate link $K$,  we define the \emph{stable penultimate
  coefficient} of $J^n_K(t)$ to be $\beta'_K:=\abs{\beta'_n}$,
for $n>1$.

Similarly, for a $B$--adequate link $K$, we define the \emph{stable
  second coefficient} of $J^n_K(t)$ to be
$\beta_K:=\abs{\beta_n}$, for $n>1$.
\end{define}

For example, in Figure \ref{fig:state-graph}, $\GRA$ is a tree. Thus,
for the link in the figure, $\beta'_K = 0$.

\begin{remark}
It is known that in general, the colored Jones polynomials $J^n_K(t)$
satisfy linear recursive relations in $n$ \cite{
 garoufalidisLe}.  In this setting, the properties stated in Theorem
\ref{thm:stabilized} can be thought of as strong manifestations of the
general recursive phenomena, under the hypothesis of adequacy.
For arbitrary knots the coefficients $\abs{\beta_n}$, $\abs{\beta'_n}$
do not, in general, stabilize. For example, for $q > p >2$,
the  coefficients $\abs{\beta_n}$, $\abs{\beta'_n}$ of the  $(p,q)$ torus link
are periodic with period 2 (see \cite{codyoliver}):  
$$\abs{\beta_{n}} \: = \: \abs{\beta_{n+2k}} \qquad \mbox{and} \qquad 
\abs{\beta'_{n}}\: = \: \abs{\beta'_{n+2k}}, \qquad \mbox{for } n \geq 2, k \in \NN.$$
See  of
\cite[Chapter 10]{fkp:gutsjp} for more discussion and questions on these periodicity phenomena.
\end{remark}

\section{State surfaces}\label{sec:2D}

In this section, we consider $2$--dimensional objects: namely, certain surfaces associated to Kauffman states. This surface is constructed as follows. Recall that a Kauffman state $\sigma$ gives rise to a collection of circles embedded in the projection plane $S^2$. Each of these circles bounds a disk in the ball \emph{below} the projection plane, where the collection of disks is unique up to isotopy in the ball.
Now, at each crossing of $D$, we connect the pair of neighboring
disks by a half-twisted band to construct a \emph{state surface} $S_\sigma \subset
S^3$ whose boundary is $K$. See Figure \ref{fig:statesurface} for an
example where $\sigma$ is the all--$A$ state.

\begin{figure}
\includegraphics{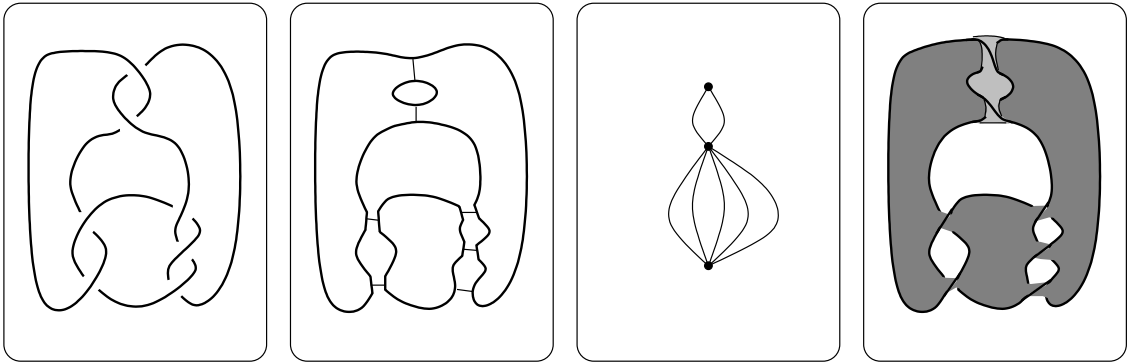}
\caption{Left to right:  A diagram.  The graphs $H_A$ and $\GA$.  The state
	surface $S_A$.}
\label{fig:statesurface}
\end{figure}

Well--known examples of state surfaces include Seifert surfaces (where
the corresponding state $\sigma$ is defined by following an
orientation on $K$) and checkerboard surfaces for alternating links
(where the corresponding state $\sigma$ is either the all--$A$ or
all--$B$ state). In this paper, we focus on the all--$A$ and all--$B$
states of a diagram, but we do not require our diagrams to be
alternating. 

Our surfaces also generalize checkerboard surfaces in the following sense. For an alternating diagram $D$, the white and shaded checkerboard surfaces are $S_A$ and $S_B$ of the all--$A$ and all--$B$ states. These surfaces can be simultaneously embedded in $S^3$ so that their intersection consists of disjoint segments, one at each crossing. Collapsing each of these segments to a point will map $S_A \cup S_B$ to the projection sphere, which is the Turaev surface $F(D)$ associated to an alternating diagram.

In our more general setting, suppose that we modify the surfaces $S_A$ and $S_B$ so that $S_A$ is constructed out of disks in the $3$--ball \emph{above} the projection plane, while $S_B$ is constructed out of disks in the $3$--ball \emph{below} the projection plane. (See Figure \ref{fig:saddle} for the boundaries of these disks.) Then, once again, $S_A \cap S_B$ will consist of disjoint segments at the crossings, and collapsing each segment to a point will map $S_A \cup S_B$ to the Turaev surface $F(D)$.  Informally, each of $S_A$ and $S_B$ forms ``half'' of the Turaev surface, just as each checkerboard surface of an alternating diagram forms ``half'' of the projection plane.

\smallskip

In general, the graph $\G_\sigma$ has the following relationship to the state surface
$S_\sigma$.

\begin{lemma}\label{lemma:ga-spine}
The graph $\mathbb{G}_\sigma$ is a spine for the surface $S_\sigma$.
\end{lemma}

\begin{proof}
By construction, $\mathbb{G}_\sigma$ has one vertex for every circe of
$s_\sigma$ (hence every disk in $S_\sigma$), and one edge for every
half--twisted band in $S_\sigma$. This gives a natural embedding of
$\mathbb{G}_\sigma$ into the surface, where every vertex is embedded
into the corresponding disk, and every edge runs through the
corresponding half-twisted band. This gives a spine for $S_\sigma$.
\end{proof}

The surfaces $S_{\sigma}$ are, in general, non-orientable (checkerboard surfaces already exhibit this phenomenon). The state graph $\mathbb{G}_\sigma$ encodes orientability via the following criterion, whose proof we leave as a pleasant exercise.

\begin{lemma}
\label{lemma:sa-orientable}
The surface $S_\sigma$ is orientable if and only if
$\mathbb{G}_\sigma$ is bipartite.
\end{lemma}

We need the following definition.

\begin{define}\label{def:essential}
Let $M$ be an orientable $3$--manifold and $S \subset M$ a properly embedded surface. 
We say that $S$ is \emph{essential}\index{essential surface} in $M$ if the boundary of a regular neighborhood of $S$, denoted $\widetilde{S}$, 
 is incompressible and boundary--incompressible. If $S$ is orientable, then $\widetilde{S}$ consists of two copies of $S$, and the definition is equivalent to the standard notion of ``incompressible and boundary--incompressible.''
 If $S$ is non-orientable, this is equivalent to $\pi_1$--injectivity of $S$, the stronger of two possible senses of incompressibility.

The state surfaces surface $S_{\sigma}$ are often non-orientable. In this case, $S^3 \cut \widetilde{S_{\sigma}}$ is the disjoint union of $M_A = S^3 \cut S_{\sigma}$ and a twisted $I$--bundle over $S_{\sigma}$.
\end{define}

Again, we are especially interested in the state surfaces of the all--$A$ and all--$B$ states. 
For these states, there is a particularly nice relationship between the state
surface $S_\sigma$ and the state graph $\mathbb{G}_\sigma$.

\begin{theorem}[Ozawa \cite{ozawa}] \label{thm:incompress} 
Let $D(K)$ be a diagram of a link $K$, and let $\sigma$ be the all--$A$ or all--$B$ state.  Then the  state 
surface $S_\sigma$ is essential in $S^3
\setminus K$ if and only if $\mathbb{G}_\sigma$ contains no 1--edge loops.
\end{theorem}

In fact, Ozawa's theorem also applies to a number of other states, which he calls $\sigma$--homogeneous \cite{ozawa}.

Ozawa proves Theorem \ref{thm:incompress} by decomposing the diagram
into tangles so that $S_\sigma$ is a Murasugi sum.  We have an alternate
proof of this result in \cite{fkp:gutsjp} that uses a decomposition of
the complement of $S_\sigma$ into topological balls.  We will discuss this
more in Section \ref{sec:polyhedra}.

\subsection{Colored Jones polynomials and slopes of state surfaces}

Garoufalidis has conjectured that for a knot $K$, the growth of the degree
of the colored Jones polynomial is related to essential surfaces in
the manifold $S^3\setminus K$ \cite{garoufalidis:jones-slopes}.  In
\cite{fkp:PAMS}, we show that this holds for $A$--adequate diagrams of
a knot $K$ and the essential surface $S_A$.  In this subsection, we
review these results.

Given $K \subset S^3$, let $M= M_K$ denote the compact 3--manifold
created when a tubular neighborhood of $K$ is removed from $S^3$.  There is a
canonical meridian--longitude basis of $H_1 (\bdy M)$, which we denote
by $\langle \mu, \lambda \rangle$.  Any properly embedded surface $(S,
\bdy S) \subset (M, \bdy M)$ has $S \cap \bdy M$ a simple closed curve
on $\bdy M$.  The homology class of $\bdy S$ in $H_1(\bdy M)$ is
determined by an element $p/q \in \QQ \cup \{ 1/0 \}$: the
\emph{slope} of $S$.  An element $p/q \in \QQ \cup \{ 1/0 \}$ is
called a \emph{boundary slope} of $K$ if there is a properly embedded
essential surface $(S, \bdy S) \subset (M, \bdy M)$, such that $\bdy
S$ is homologous to $p \mu + q \lambda \in H_1 (\bdy M)$.  Hatcher has
shown that every knot $K \subset S^3$ has finitely many boundary
slopes \cite{hatcher}.

Let $j(n)$ denote the highest degree of $J_K(n, t)$ in $t$, and let
$j'(n)$ denote the lowest degree.  Consider the sequences
$$js_K:= \left\{{{ 4j(n)}\over {n^2} } \: : \: n > 0\right\} \quad
\mbox{and} \quad js'_K:= \left\{ {{ 4j'(n)}\over {n^2}} \: : \: n > 0
\right\}.$$

Garoufalidis has conjectured \cite{garoufalidis:jones-slopes}  that for
each knot $K$, every cluster point (i.e., every limit of a
subsequence) of $js_K$ or $js^*_K$ is a boundary slope of $K$.  In
\cite{fkp:PAMS}, the authors proved this is true for $A$--adequate
knots, and the boundary slope comes from the incompressible surface
$S_A$.  This is the content of the following theorem.

\begin{figure}[h]
\includegraphics{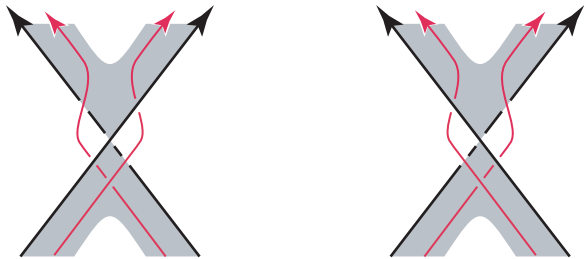}
\caption{Left: a positive crossing, and a piece of $S_B$ near the crossing.  Locally,
this crossing contributes $+2$ to the slope of $S_B$, and makes no contribution to the slope of $S_A$. Right: a negative crossing contributes $-2$ to the slope of $S_A$, and makes no contribution to the slope of $S_B$.
}
\label{fig:local-slope}
\end{figure}

\begin{theorem}[\cite{fkp:PAMS}]
\label{thm:slopes}
Let $D$ be an $A$--adequate diagram of a knot $K$ and let $b(S_A) \in
\ZZ$ denote the boundary slope of the essential surface $S_A$.  Then
$$\lim_{n\to \infty} {\frac {4 j'(n)}{n^2}}= b(S_A)=-2c_-,$$ where
$c_-$ is the number of negative crossings in $D$. (See Figure \ref{fig:local-slope}, right.)

Similarly, if $D$ is a $B$--adequate diagram of a knot $K$, let
$b(S_B) \in {\ZZ}$ denote the boundary slope of the essential surface
$S_B$.  Then
$$\lim_{n\to \infty} {\frac {4 j(n)}{n^2}}= b(S_B)=-2c_+,$$
where $c_+$ is the number of
positive crossings in $D$.  (See Figure \ref{fig:local-slope}, left.)
\end{theorem}

Additional families of knots for which the conjecture is true are
given by Garoufalidis \cite{garoufalidis:jones-slopes} and more
recently by Dunfeld and Garoufalidis \cite{dunfield-garoufalidis:2fusion}.

\section{Cutting along the state surface}\label{sec:3D}

In this section, we focus on the  3--manifold formed by cutting along the state surface $S_A$. Using its $3$--dimensional structure, we will relate the hyperbolic geometry of $S^3 \setminus K$ to the Jones and colored Jones polynomials of $K$.

\subsection{Geometry and topology of the state surface complement}

\begin{define}\label{def:cut}
  Let $K\subset S^3$ be a link, and $S_A$ the all--$A$ state surface.
  We let $M$ denote the link complement, $M = S^3 \setminus K$, and we
  let $M_A := M\cut S_A$\index{$M\cut S_A$} denote the path--metric closure
  of $M \setminus S_A$.  Note that $M_A = (S^3\setminus K)\cut S_A$ is
  homeomorphic to $S^3\cut S_A$, obtained by removing a regular
  neighborhood of $S_A$ from $S^3$.

We will refer to $P=\bdy M_A \cap \bdy M$ as the \emph{parabolic
  locus}\index{parabolic locus} of $M_A$; it consists of annuli. The
remaining, non-parabolic boundary $\bdy M_A \setminus \bdy M$ is the
unit normal bundle of $S_A$.
\end{define}

Our goal is to use the state graph $\GA$ to understand the topological structure of $M_A$. One result along these lines is a straightforward characterization of when $S_A$ is a fiber surface for $S^3 \setminus K$, or equivalently when $M_A$ is an $I$--bundle over $S_A$.

\begin{theorem} \label{thm:fiber-tree}
Let $D(K)$ be any link diagram, and let $S_A$ be the spanning surface
determined by the all--$A$ state of this diagram. Then the following
are equivalent:
\begin{enumerate}
\item The reduced graph $\GRA$ is a tree.
\item  $S^3 \setminus K$ fibers over $S^1$, with fiber $S_A$.
\item $M_A = S^3 \cut S_A$ is an $I$--bundle over $S_A$.
\end{enumerate}
\end{theorem}

For example, in the link diagram depicted in Figure \ref{fig:statesurface}, the graph $\GRA$ is a tree with two edges.  Thus the state surface $S_A$ shown in  Figure \ref{fig:statesurface} is a fiber in $S^3 \setminus K$.

In general, one may apply the annulus version of the JSJ
decomposition theory \cite{jaco-shalen, johannson} to cut $M_A$ into
three types of pieces: $I$--bundles over sub-surfaces of $S_A$,
Seifert fibered spaces that are solid tori, and \emph{guts}, i.e. the
portion that admits a hyperbolic metric. 

The pieces of the JSJ decomposition give significant information about
the manifold $M_A$.  For example, if $\guts(M_A) = \emptyset$, then
$M_A$ is a union of $I$--bundles and solid tori.  Such an $M_A$ is
called a \emph{book of $I$--bundles}, and the surface $S_A$ is called
a \emph{fibroid} \cite{culler-shalen}.

The guts of $M_A$ are a good measurement of topological complexity. 
To express this more precisely, we need the following
definition.  

\begin{define}\label{def:neg-euler}
  Let $Y$ be a compact cell complex with connected components $Y_1,
  \ldots, Y_n$.  Let $\chi(\cdot)$ denote the Euler characteristic.
  This can be split into positive and negative parts, with notation
  borrowed from the Thurston norm \cite{thurston:norm}:
$$
\chi_+(Y) = \sum_{i=1}^{n} \max \{  \chi(Y_i), \, 0 \}, \qquad 
\negeul(Y) = \sum_{i=1}^{n} \max \{ - \chi(Y_i), \, 0 \}.
$$
Note that $\chi(Y) = \chi_+ (Y) - \negeul(Y)$.  In the case that $Y =
\emptyset$, we have $\chi_+(\emptyset) = \negeul(\emptyset) = 0$.
\end{define}

The negative Euler characteristic $\negeul(\guts(M_A))$ serves as a
useful measurement of how far $S_A$ is from being a fiber or a fibroid
in $S^3 \setminus K$.  In addition, it relates to the volume of $M_A$.
The following theorem was proved by
Agol, Storm, and Thurston.

\begin{theorem}[Theorem 9.1 of \cite{ast}]\label{thm:ast-estimate}
Let $M$ be finite--volume hyperbolic $3$--manifold, and let $S \subset
M$ be a properly embedded essential surface.  Then
$$\vol(M) \:\geq \: v_8\, \negeul(\guts(M\cut S)),$$
where $v_8 = 3.6638...$ is the volume of a regular ideal octahedron.
\end{theorem}

We apply Theorem \ref{thm:ast-estimate} to the essential surface $S_A$
for a prime, $A$--adequate diagram of a hyperbolic link.  In order to
do so, we develop techniques for determining the Euler characteristic
of the guts of $S_A$.  We find that it can be read off of a diagram of the
link.

\begin{theorem}\label{thm:guts-general}
Let $D(K)$ be an $A$--adequate diagram, and let $S_A$ be the essential
spanning surface determined by this diagram. Then
$$\negeul( \guts(S^3 \cut S_A))= \negeul (\GRA) - || E_c||,$$ where
$|| E_c|| \geq 0$ is a diagrammatic quantity.
\end{theorem}

The quantity $|| E_c ||$ is the number of \emph{complex essential
  product disks} (EPDs).  We will give its definition and examples in
the next subsection.  For now, we point out that in many cases, the
quantity $ || E_c||$ vanishes. For example, this happens for
alternating links \cite{lackenby:volume-alt}, as well as for most
Montesinos links \cite[Corollary 9.21]{fkp:gutsjp}.

When we combine Theorems \ref{thm:ast-estimate} and
\ref{thm:guts-general}, and recall that $S^3\cut S_A$ is
homeomorphic to $(S^3\setminus K)\cut S_A$, we obtain
$$\vol(S^3
\setminus K) \: \geq \: v_8\, \negeul(\guts(S^3\cut S_A)) \: =\:
\negeul(\GRA)-|| E_c||,$$ 
where the equality comes from Theorem
\ref{thm:guts-general}. This leads to the following.

\begin{theorem}\label{thm:volume}
  Let $D=D(K)$ be a prime $A$--adequate diagram of a hyperbolic link
  $K$.  Then
$$\vol(S^3 \setminus K) \: \geq \: v_8\, (\negeul(\GRA)-|| E_c||),$$
where $|| E_c||$ is the same diagrammatic quantity as in the statement of
  Theorem \ref{thm:guts-general}, and $v_8 = 3.6638...$ is the volume
  of a regular ideal octahedron.
\end{theorem}

There is a symmetric result for $B$--adequate diagrams.

\subsection{Essential product disks}\label{subsec:epds}

We now describe the quantity $||E_c||$ more carefully, and discuss how it relates to  the graph $\GA$.  To begin, we review some
terminology.

\begin{define}\label{def:epd}
Let $M$ be a $3$--manifold with boundary, and with prescribed parabolic locus consisting of annuli. 
An \emph{essential product disk} in $M$, or \emph{EPD} for short, is a properly embedded disk
whose boundary has geometric intersection number $2$ with the parabolic locus. Note that an EPD is an $I$--bundle over an interval. See Figure \ref{fig:EPD} for an example of such a disk in $M_A = S^3 \cut S_A$.

\begin{figure}
\psfrag{S}{$S_A$}
\includegraphics{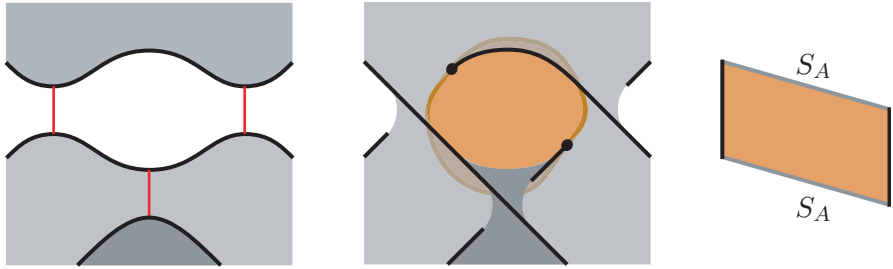}
\caption{An EPD in $M_A$ containing a single 2-edge loop in $\GA$, with
  edges in different twist regions in the link diagram.}
\label{fig:EPD}
\end{figure}

If $B$ is an $I$--bundle in $M$, we say that a collection $\{D_1, \ldots, D_n\}$ of disjoint EPDs \emph{spans} $B$ if their complement in $B$ is a disjoin union of solid tori and $3$--balls.
\end{define}

Essential product disks are integral to understanding the size of $\guts(M_A)$.  In particular, the proof of Theorem \ref{thm:guts-general} requires calculating the Euler characteristic of all the $I$--bundle components in the JSJ decomposition of $M_A = S^3 \cut S_A$. To do this, we show that each component of the $I$--bundle is spanned by EPDs, and find a particular spanning set. (See Theorem \ref{thm:epd-span} in Section \ref{sec:polyhedra}.)

\begin{define}\label{primetwist}
Two crossings in $D$ are defined to be \emph{twist equivalent} if
there is a simple closed curve in the projection plane that meets $D$
at exactly those two crossings. The diagram is called \emph{twist
  reduced} if every equivalence class of crossings is a \emph{twist
  region} (a chain of crossings between two strands of $K$).  The
number of equivalence classes is denoted $t(D)$, the \emph{twist
  number} of $D$.
\label{def:prime-diagram}
\end{define}

Every twist region in $D(K)$ with at least two crossings gives rise to
EPDs.  For instance, in Figure \ref{fig:EPD-twist}, there are three crossings 
in the twist region.  The boundary
of each EPD shown lies on the state surface $S_A$, and crosses the
knot diagram exactly twice.  Note there are two EPDs that encircle one bigon each, 
and one EPD that encircles two
bigons. Any two of these will suffice in a spanning set.

\begin{figure}
  \begin{center}
    \begin{tabular}{ccccc}
       \includegraphics{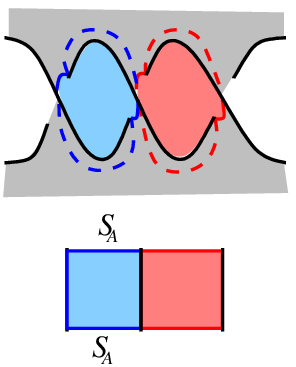}
      & \hspace{.2in} & 
      \includegraphics{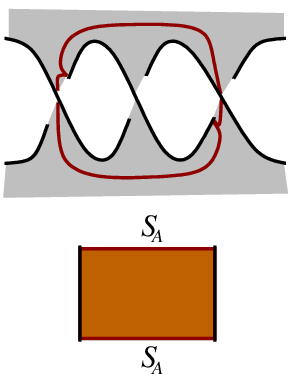}
    \end{tabular}
  \end{center}
  \caption{Shown are three EPDs in a twist region.}
  \label{fig:EPD-twist}
  \end{figure}

The essential product disk in Figure \ref{fig:EPD} does not lie in a single twist
region. For another example that does not lie in a single twist
region, see Figure \ref{fig:twoloops}. Note that in each of Figures
\ref{fig:EPD}, \ref{fig:EPD-twist}, and \ref{fig:twoloops}, the EPD
can be naturally associated to one or more $2$--edge loops in the
state graph $\GA$.

\begin{figure}
\includegraphics[height=1.8in]{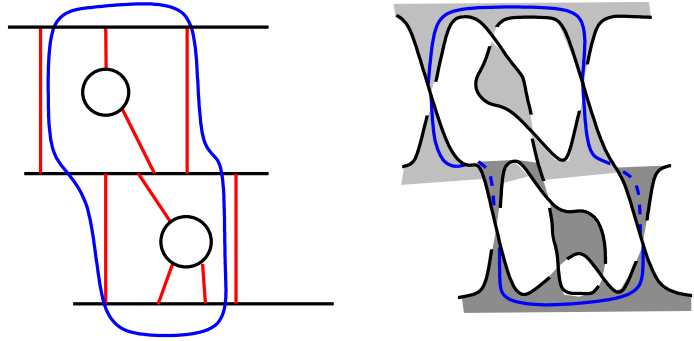}
\caption{An EPD in $M_A$ containing two  2-edge loops of $\GA$.}
\label{fig:twoloops}
\end{figure}

In Figure \ref{fig:twoloops}, the EPD exhibits more complicated
behavior than in the other examples, in that it bounds nontrivial
portions of the graph $H_A$.  We call an EPD that bounds nontrivial
portions of $H_A$ on both sides a \emph{complex} EPD. The minimal
number of complex EPDs in the spanning set of the maximal $I$--bundle
of $M_A$ is denoted $||E_c||$, and is exactly the correction term in
Theorem \ref{thm:guts-general}.  By analyzing EPDs, we show the
following.

\begin{prop}\label{prop:no2loops}
  If $D$ is prime and $A$--adequate, such that every 2--edge loop in
  $\GA$ has edges belonging to the same twist region, then $||E_c|| =
  0$.  Hence
$$\vol(S^3 \setminus K) \: \geq \: v_8 \, (\negeul(\GRA)).$$
\end{prop}

\subsection{Volume estimates}
One family of knots and links that satisfies Proposition \ref{prop:no2loops} is that of alternating links
 \cite{lackenby:volume-alt}. If $D=D(K)$ is a prime, twist--reduced alternating link diagram, then it is both $A$-- and
$B$--adequate, and for each 2--edge loop in $\GA$ or $\GB$, both edges
belong to the same twist region.  Theorem  \ref
{thm:volume} gives lower bounds on volume in terms of both
$\negeul(\GRA)$ and $\negeul(\GRB)$. By averaging these two lower
bounds, one recovers Lackenby's lower bound on the volume of hyperbolic
alternating links, in terms of the twist number $t(D)$.

\begin{theorem}[Theorem 2.2 of \cite{ast}]
  Let $D$ be a reduced alternating diagram of a hyperbolic link $K$.
  Then
$$\frac{ v_8}{2} \, (t(D)-2)  \: \leq \:\vol(S^3 \setminus K) \: < \:
  10v_3\, (t(D)-1),$$  
where $v_3 = 1.0149...$ is the volume of a regular ideal tetrahedron  
and $v_8 = 3.6638...$ is the volume of a regular ideal octahedron. 
\label{thm:lacalternating}
\end{theorem}

Theorem \ref{thm:volume} greatly expands the list of manifolds for
which we can compute explicitly the Euler characteristic of the guts,
and can be used to derive results analogous to Theorem
\ref{thm:lacalternating}.  As a sample, we state the following.

\begin{theorem}\label{thm:positive-volume}
Let $D(K)$ be a diagram of a hyperbolic link $K$, obtained as the
closure of a positive braid with at least three crossings in each
twist region. Then
$$\frac{ 2v_8}{3} \, t(D) \: \leq \:\vol(S^3 \setminus K) \: < \: 10v_3(t(D)-1),$$ 
where $v_3 = 1.0149...$ is the volume of a regular ideal tetrahedron  
and $v_8 = 3.6638...$ is the volume of a regular ideal octahedron. 
\end{theorem}

Observe that the multiplicative constants in the upper and lower
bounds differ by a rather small factor of about $4.155$.

We obtain similarly tight two--sided volume bounds for Montesinos
links, using these guts techniques \cite{fkp:gutsjp}.

\begin{theorem} \label{thm:monte-volume}
  Let $K \subset S^3$ be a Montesinos link with a reduced Montesinos
  diagram $D(K)$.  Suppose that $D(K)$ contains at least three
  positive tangles and at least three negative tangles.  Then $K$ is a
  hyperbolic link, satisfying
  \begin{equation*}
    \frac{v_8}{4} \, \left( t(D) - \#K \right) \:
    \leq \: \vol(S^3 \setminus K) \:
    < \: 2 v_8 \, t(D),
  \end{equation*}
  where $v_8 = 3.6638...$ is the volume of a regular ideal octahedron
  and $\# K$ is the number of link components of $K$.  The upper bound
  on volume is sharp.
\end{theorem}

Similar results using different techniques have been obtained by the
authors in \cite{fkp:filling, fkp:conway, fkp:farey}, and by Purcell
in \cite{purcell:volumes}.

\subsection{Relations with the colored Jones polynomial}

The main results of \cite{fkp:gutsjp} explore the idea that the stable
coefficient $\beta_K'$ does an excellent job of measuring the
geometric and topological complexity of the manifold $M_A = S^3 \cut
S_A$. (Similarly, $\beta_K$ measures the complexity of $M_B = S^3 \cut
M_B$.)

For instance, note that we have $\abs{\beta'_K}=1 - \chi(\GRA)=0$
exactly when $\chi(\GRA)=1$, or equivalently $\GRA$ is a tree.
Thus it follows from Theorem \ref{thm:fiber-tree}  that
$\beta'_K$ is exactly the obstruction to $S_A$ being a fiber.

\begin{corollary} \label{cor:beta-fiber}
  For an $A$--adequate link $K$, the following are equivalent:
  \begin{enumerate}
  \item $\beta'_K=0$. 
  \item For \emph{every} $A$--adequate diagram of $D(K)$, $S^3
    \setminus K$ fibers over $S^1$ with fiber the corresponding state
    surface $S_A = S_A(D)$.
  \item For \emph{some} $A$--adequate diagram $D(K)$, $M_A = S^3 \cut
    S_A$ is an $I$--bundle over $S_A(D)$.
  \end{enumerate}
\end{corollary}

Similarly, $\abs{\beta_K'} = 1$ precisely when $S_A$ is a fibroid of a
particular type \cite[Theorem 9.18]{fkp:gutsjp}.  In general, the JSJ
decomposition of $M_A$ contains guts: the non-trivial hyperbolic
pieces. In this case, $\abs{\beta_K'}$ measures the complexity of the
guts together with certain complicated parts of the maximal
$I$--bundle of $M_A$.

\begin{theorem}  \label{thm:no2loops}
  Suppose $K$ is an $A$--adequate link whose stable colored Jones
  coefficient is $\beta_K' \neq 0$.  Then, for every $A$--adequate
  diagram $D(K)$,
$$\negeul( \guts(M_A)) + || E_c || \: = \: \abs{\beta'_K} - 1. $$
  Furthermore, if $D$ is prime and every 2--edge loop in $\GA$ has
  edges belonging to the same twist region, then $||E_c|| = 0$ and
$$\negeul( \guts(M_A)) \: = \: \abs{\beta'_K}- 1. $$
\end{theorem}

The volume conjecture of Kashaev and Murakami--Murakami
\cite{kashaev:vol-conj, murakami:vol-conj} states that all hyperbolic
knots satisfy
$$2\pi\lim_{n\to \infty} { \frac{\log \abs {J^n_K(e^{2\pi i / n} ) }}
  {n}}=\vol(S^3 \setminus K).$$

If the volume conjecture is true, then for large $n$, there would be a
relation between the coefficients of $J^n_K(t)$ and the volume of the
knot complement.  In recent years, articles by Dasbach and Lin and the
authors have established relations for several classes of knots
\cite{dasbach-lin:volumish, fkp:filling, fkp:conway, fkp:farey}.
However, in all these results, the lower bound involved first showing
that the Jones coefficients give a lower bound on twist number, then
showing twist number gives a lower bound on volume.  Each of these
steps is known to fail outside special families of knots
\cite{fkp:farey, fkp:coils}.  Moreover, the two-step argument is
indirect, and the constants produced are not sharp.  By contrast, in
\cite{fkp:gutsjp}, we bound volume below in terms of $\negeul(\guts)$,
which is directly related to colored Jones coefficients.  This
yields sharper lower bounds on volumes, along with a more intrinsic
explanation for why these lower bounds exist.  For instance, Theorems  \ref{thm:positive-volume} and \ref{thm:monte-volume} have
the following corollaries.

\begin{corollary} \label{cor:positive-vol-jones}
Suppose that a hyperbolic link $K$ is the closure of a positive braid
with at least three crossings in each
twist region. Then
$$v_8 \, ( \abs{\beta'_K}-1 )\: \leq \:\vol(S^3 \setminus K) \: < \: 15 v_3 \, (\abs{\beta'_K}-1) - 10 v_3,$$ 
where $v_3 = 1.0149...$ is the volume of a regular ideal tetrahedron  
and $v_8 = 3.6638...$ is the volume of a regular ideal octahedron.
\end{corollary}

\begin{corollary}\label{cor:jones-volumemonte}
  Let $K \subset S^3$ be a Montesinos link with a reduced Montesinos
  diagram $D(K)$.  Suppose that $D(K)$ contains at least three
  positive tangles and at least three negative tangles.  Then $K$ is a
  hyperbolic link, satisfying
$$ v_8 \left( \max \{ \abs{\beta_K}, \abs{\beta'_K} \} -1 \right) \:
  \leq \: \vol(S^3 \setminus K) \: < \: 4 v_8 \left( \abs{\beta_K} +
  \abs{\beta'_K} -2 \right) + 2 v_8 \, (\# K),$$ where $\#K$ is the
  number of link components of $K$.
\end{corollary}

\section{A worked example}\label{sec:example}
In this section, we illustrate several of the above theorems on the
two-component link  of Figure \ref{fig:effie}. The figure also shows the graphs
$H_A$, $\GA$, and $\GRA$ for this link diagram.

\begin{figure}
\includegraphics[height=1.2in]{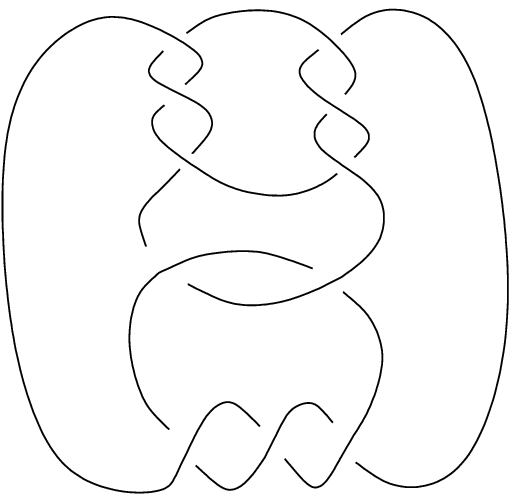}
\includegraphics[height=1.2in]{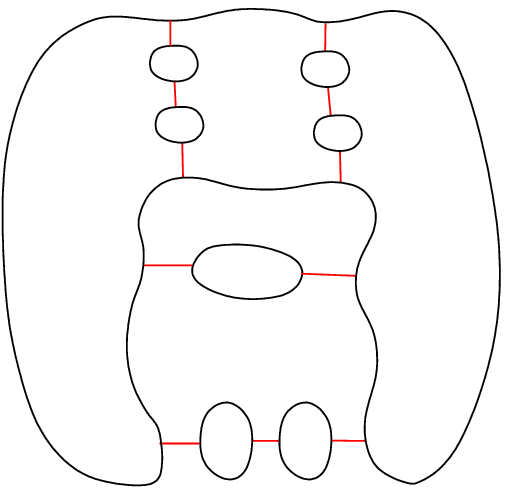}
\hspace{0.1in}
\includegraphics[height=1.2in]{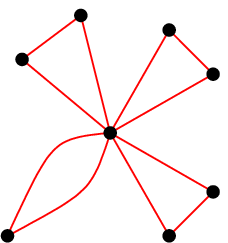}
\hspace{0.1in}
\includegraphics[height=1.2in]{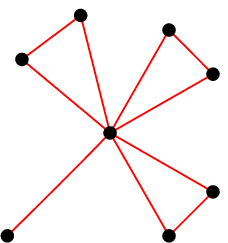}
\hspace{0.1in}
\caption{Diagram $D(K)$ of a two-component link, and graphs $H_A$,
  $\GA$, and $\GRA$. All of the discussion in Section \ref{sec:example} pertains to this link.}
\label{fig:effie}
\end{figure}

\begin{figure}
\includegraphics[height=1.4in]{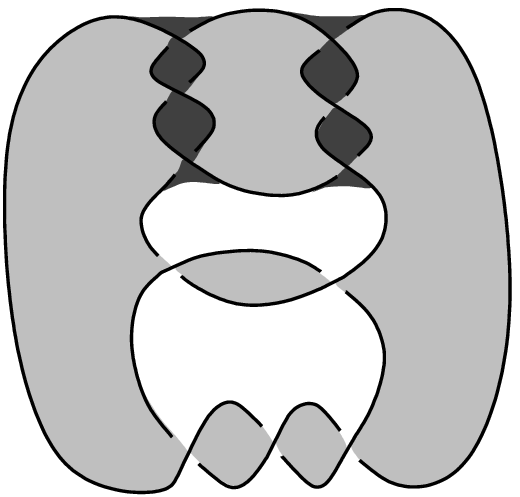}
\hspace{.2in}
\includegraphics[height=1.4in]{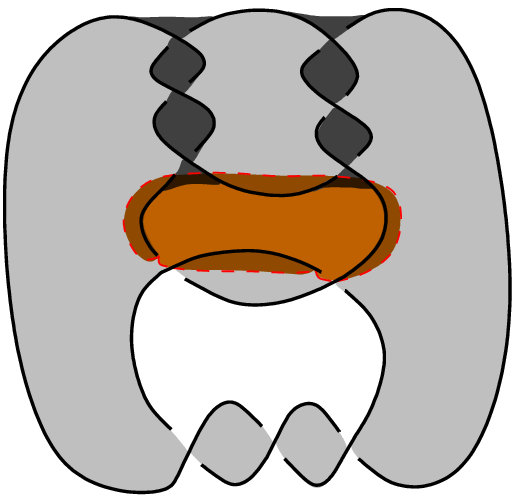}
\caption{Left: the state surface $S_A$ of the link of Figure
  \ref{fig:effie}. Right: the single EPD in $M_A$ lies below the surface $S_A$, and its boundary intersects $K$ at two points in the center of the figure.}
\label{fig:effie-EPD}
\end{figure}

In this example, it turns out that the manifold $M_A = S^3 \setminus
S_A$ contains a single essential product disk. This disk $D$ is shown
in Figure \ref{fig:effie-EPD}. Observe that this lone EPD corresponds
to the single $2$--edge loop in $\GA$. Note as well that collapsing
two edges of $\GA$ to a single edge of $\GRA$ changes the Euler
characteristic by $1$, while cutting $M_A$ along disk $D$ also changes
the Euler characteristic by $1$. Thus
$$\negeul(\GA) \: = \:  - \chi(\GA) = 3, \qquad \negeul (\GRA) \: = \:
- \chi(\GRA) = 2.$$

On the $3$--manifold side, recall that $\GA$ is a spine of $S_A$. Thus Alexander duality gives
$$\negeul( M_A)  \: = \: \negeul(S^3 \cut S_A) \: = \: \negeul(S_A) \: = \: \negeul(\GA) = 3.$$
Because the maximal $I$--bundle of $M_A$ is spanned by the single disk $D$, we have
$$\negeul(\guts(M_A)) \: = \: \negeul( M_A) - 1  \: = \: 2 \: = \:  \negeul (\GRA),$$
exactly as predicted by Theorem \ref{thm:guts-general} with $||E_c|| = 0$.

By Theorem \ref{thm:ast-estimate}, $\negeul(\guts(M_A))$  also gives a lower bound on the hyperbolic volume of $S^3 \setminus K$. In this example,
$$\negeul(\guts(M_A)) \:= \: \negeul(\GRA) \: =\: \abs{\beta'_K} - 1 \: = \: 2,$$
so the lower bound is $2 v_8 \approx 7.3276$. Meanwhile, the actual hyperbolic volume of the link in Figure \ref{fig:effie} is $\vol(S^3 \setminus K) \approx 11.3407$.

\begin{figure}
\includegraphics[height=1.4in]{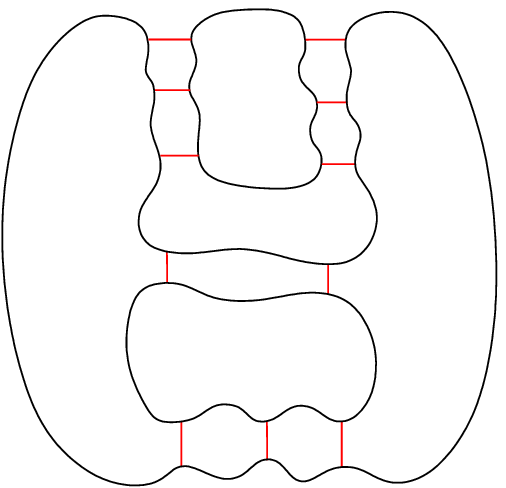}
\hspace{0.3in}
\includegraphics[height=1.2in]{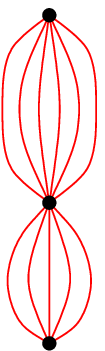}
\hspace{0.3in}
\includegraphics[height=1.2in]{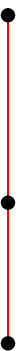}
\hspace{0.3in}
\includegraphics[height=1.4in]{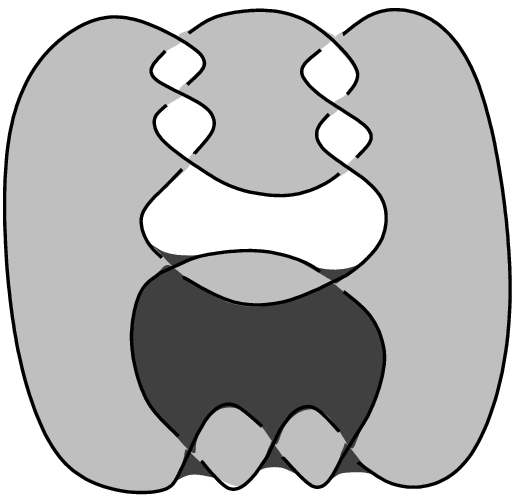}
\caption{The graphs $H_B$, $\GB$, and $\GRB$ for the link of Figure \ref{fig:effie}, and the corresponding surface $S_B$.}
\label{fig:effie-B}
\end{figure}

Turning to the all--$B$ resolution, Figure \ref{fig:effie-B} shows the graphs $H_B$, $\GB$, and $\GRB$, as well as the state surface $S_B$. This time, $\GRB$ is a tree, thus Theorem \ref{thm:fiber-tree} (applied to a reflected diagram) implies $S_B$ is a fiber. 

One important point to note is that even though $S^3 \setminus K$ is fibered, it does contain surfaces (such as $S_A$) with quite a lot of guts. Conversely, having $\abs{\beta'_K} > 0$, as we do here, only means that the surface $S_A$ is not a fiber -- it does \emph{not} rule out $S^3 \setminus K$ being fibered in another way, as indeed it is.


\section{A closer look at the polyhedral decomposition}\label{sec:polyhedra}
To prove all the results that were surveyed in Section \ref{sec:3D}, we cut $M_A$ along a collection of
disks, to obtain a decomposition of $M_A$ into ideal polyhedra.
Here, a (combinatorial) \emph{ideal polyhedron} is a 3--ball with a graph on its boundary, such
that complementary regions of the graph are simply connected, and the
vertices have been removed (i.e. lie at infinity).

Our decomposition is a generalization of Menasco's well--known polyhedral
decomposition  \cite{menasco:polyhedra}.  Menasco's work uses a link diagram to
decompose any link complement into ideal polyhedra.
When the diagram is alternating, the resulting polyhedra have several
nice properties: they are checkerboard colored, with 4--valent
vertices, and a well--understood gluing.  For alternating diagrams,
our polyhedra will be exactly the same as Menasco's.  More generally,
we will see that our polyhedral decomposition of $M_A$ also has a
checkerboard coloring and 4--valent vertices.

\subsection{Cutting along disks}

To begin, we need to visualize the state surface $S_A$ more carefully.
We constructed $S_A$ by first, taking a collection of disks bounded by
state circles, and then attaching bands at crossings.  Recall we
ensured the disks were below the projection plane.  We visualize the
disks as \emph{soup cans}.  That is, for each, a long cylinder runs
deep under the projection plane with a disk at the bottom.  Soup cans
will be nested, with outer state circles bounding deeper, wider soup
cans.  Isotope the diagram so that it lies on the projection plane,
except at crossings which run through a crossing ball.  When we are
finished, the surface $S_A$ lies below the projection plane, except
for bands that run through a small crossing ball.

In Figure \ref{fig:multiballons}, the state surface for this example
is shown lying below the projection plane (although soup cans have
been smoothed off at their sharp edges in this figure).

\begin{figure}
	\includegraphics[height=1.2in]{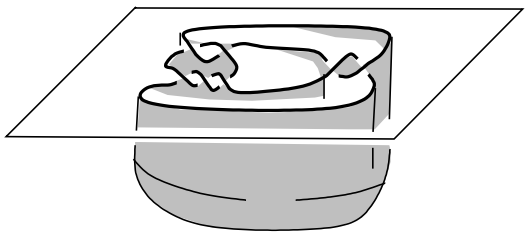}
  \caption{The surface $S_A$ is hanging below the projection plane.}
\label{fig:multiballons}
\end{figure}

We cut $M_A = S^3 \cut S_A$ along disks.  These disks come from
complementary regions of the graph of the link diagram on the
projection plane.  Notice that each such region corresponds to a
complementary region of the graph $H_A$, the graph of the
$A$--resolution.  To form the disk that we cut along, we isotope the
disk by pushing it under the projection plane slightly, keeping its
boundary on the state surface $S_A$, so that it meets the link a
minimal number of times.  Indeed, because $S_A$ itself lies on or
below the projection plane except in the crossing balls, we can push
the disk below the projection plane everywhere except possibly along
half--twisted rectangles at the crossings.  By further isotopy we can
arrange each disk so that its boundary runs meets the link only inside
the crossing ball. These isotoped disks are called \emph{white disks}.

For each region of the complement of $H_A$, we have a white disk that
meets the link only in crossing balls, and then only at
under--crossings.  The disk lies slightly below the projection plane
everywhere.  Figure \ref{fig:disk} gives an example.

\begin{figure}
  \begin{center}
    \begin{tabular}{ccccc}
       \includegraphics{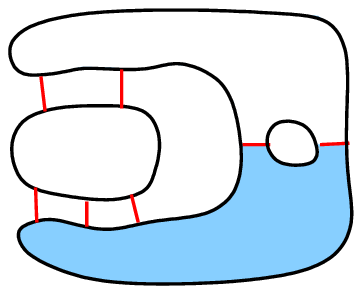}
      & \hspace{.2in} & 
      \includegraphics{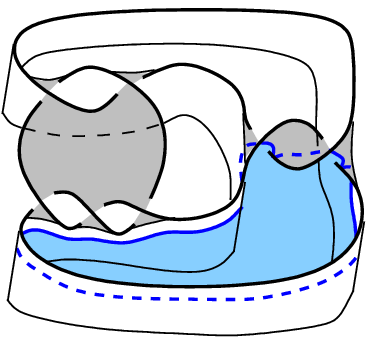}
    \end{tabular}
  \end{center}
  \caption{A region of $H_A$ together with the corresponding white
    disk lying just below the projection plane, with boundary (dashed
    line) on underside of shaded surface. }
  \label{fig:disk}
\end{figure}

Some of the white disks will not meet the link at all.  These disks
are isotopic to soup cans on $S_A$; that is, they are \emph{innermost disks}.  We will remove all such
white disks from consideration.  When we do so, the collection of all
remaining white disks, denoted $\CalD$, consists of those with
boundary on the state surface $S_A$ and on the link $K$.

It is actually straightforward to see that the components of
$M_A \cut \CalD$ are 3--balls, as follows.  There will be a single
component above the projection plane.  Since we cut along each region
of the projection graph, either along a disk of $\CalD$ or an
innermost soup can, this component above the projection plane must be
homeomorphic to a ball.  As for components which lie below the
projection plane, these lie between soup can disks.  Since any such
disk cuts the 3--ball below the projection plane into 3--balls, these
components must also each be homeomorphic to 3--balls.
In fact, we know more:

\begin{theorem}[Theorem 3.12 of \cite{fkp:gutsjp}]
\label{thm:3-ball}
Each component of $M_A \cut \CalD$ is an ideal polyhedron with
4--valent ideal vertices and faces colored in a checkerboard fashion:
the \emph{white faces} are the disks of $\CalD$, and the \emph{shaded
faces} contain, but are not restricted to, the innermost disks.
\end{theorem}

The edges of the ideal polyhedron are given by the intersection of
white disks in $\CalD$ with (the boundary of a regular neighborhood
of) $S_A$.  Each edge runs between strands of the link.
The ideal vertices lie on the torus boundary of the tubular
neighborhood of a link component.
The regions (faces) come from white disks (white faces) and portions
of the surface $S_A$, which we shade.

Note that each edge bounds a white disk in $\CalD$ on one side, and a
portion of the shaded surface $S_A$ on the other side.  Thus, by
construction, we have a checkerboard coloring of the 2--dimensional
regions of our decomposition.  Since the white regions are known to be
disks, showing that our 3--balls are actually polyhedra amounts to
showing that the shaded regions are also simply connected.  Showing
this requires some work, and the hypothesis of $A$--adequacy is
heavily used. The interested reader is referred to \cite[Chapter
  3]{fkp:gutsjp} for the details.

\subsection{Combinatorial descriptions of the polyhedra}\label{sec:combinatorial}

We need a simpler description of our polyhedra than that afforded by
the 3--dimensional pictures of the last subsection.  We obtain
2--dimensional descriptions in terms of how the white and shaded faces
are super-imposed on the projection plane, and how these faces
interact with the planar graph $H_A$.  These descriptions are the
starting point for our proofs in \cite{fkp:gutsjp}.  In this
subsection we briefly highlight the main characteristics of these
descriptions.

Note that in the figures, we often use different colors to indicate
different shaded faces.  All these colored regions come from the
surface $S_A$.

\subsubsection{Lower polyhedra}

The lower polyhedra come from regions bounded between soup cans.
Recall that $s_A$ denotes the union of state circles of the all--$A$
resolution (i.e. without the added edges of $H_A$ corresponding to
crossings).  As a result, the regions bounded by soup cans will be in
one--to--one correspondence with non-trivial components of the
complement of $s_A$.

Given a lower polyhedron, let $R$ denote the corresponding non-trivial
component of the complement of $s_A$.  The white faces of the
polyhedron will correspond to the non-trivial regions of $H_A$ in
$R$.  Since these white faces lie below the projection plane, except
in crossing balls, the only portion of the knot that is visible from inside a lower polyhedron  is
a small segment of a crossing ball.  This results in the following combinatorial description.
\begin{itemize}
  \item  Ideal edges of the lower polyhedra run from crossing to crossing.
  \item  Ideal vertices correspond to crossings. At each crossing, two
    ideal edges bounding a disk from one non-trivial region of $H_A$
    meet two ideal edges bounding a disk from another non-trivial region
    (on the opposite side of the crossing). Thus the vertices are 4--valent.
  \item  Shaded faces correspond exactly to soup cans.
\end{itemize}

As a result, each lower polyhedron is combinatorially identical to  the checkerboard polyhedron of an
alternating sub-diagram of $D(K)$, where the sub-diagram corresponds to region
$R$.   This is illustrated in Figure \ref{fig:lowerpoly}, for the knot diagram of Figure \ref{fig:statesurface}.

\begin{figure}
  \begin{center}
    \begin{tabular}{ccccc}
      
      \includegraphics{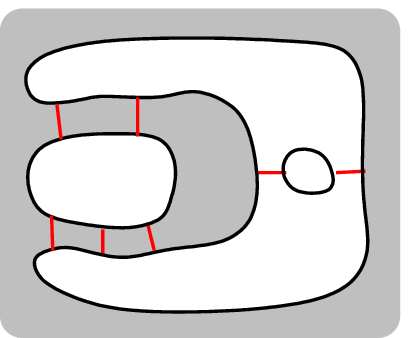} & \hspace{.2in}
      & \includegraphics{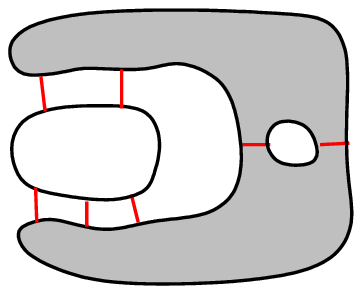}\\
      \includegraphics{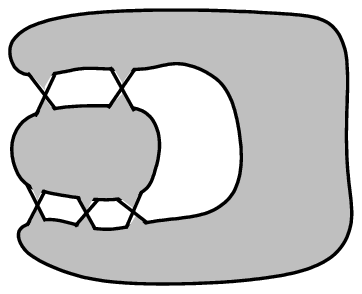} & \hspace{.2in} & 
      \includegraphics{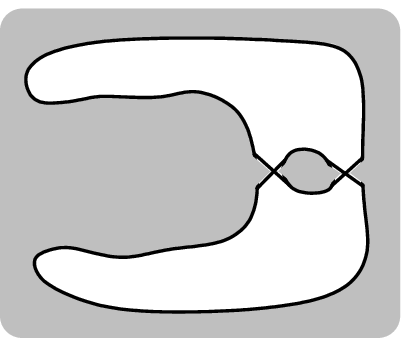}
    \end{tabular}
  \end{center}
  \caption{Top row: The two non-trivial regions of the graph $H_A$ of
    Figure \ref{fig:statesurface}. Second row: The corresponding lower
    polyhedra. }
  \label{fig:lowerpoly}
\end{figure}

Schematically, to sketch a lower polyhedron, start by drawing a
portion of $H_A$ which lies inside a nontrivial region of the
complement of $s_A$.  Mark an ideal vertex at the center of each
segment of $H_A$.  Connect these dots by edges bounding white disks,
as in Figure \ref{fig:lower-schematic}.

\begin{figure}
	\includegraphics{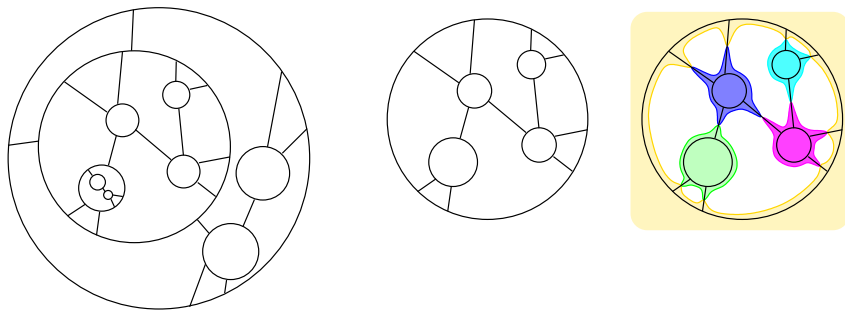}
\caption{Left to right: An example graph $H_A$.  A subgraph
  corresponding to a region of the complement of $s_A$.  White and
  shaded faces of the corresponding lower polyhedron\index{lower
    polyhedra!example}.}
\label{fig:lower-schematic}
\end{figure}

\subsubsection{The upper polyhedron}

On the upper polyhedron, ideal vertices correspond to strands of the
link visible from inside the upper 3--ball.  Since we cut along white
faces and the surface $S_A$, both of which lie below the projection plane 
except at crossings, the upper polyhedron can ``see''  the entire link diagram except for small portions cut off at each undercrossing.  

Thus, ideal vertices of the upper polyhedron correspond to strands of the link between
undercrossings.  In Figure \ref{fig:under-crossing}, the surface $S_A$
is shown in green and gold.  There are four ideal edges meeting a
crossing, labeled $e_1$ through $e_4$ in this figure.  There are
actually three ideal vertices in the figure: one is unlabeled,
corresponding to the strand running over the crossing, and two are
labeled $v_1$ and $v_2$, corresponding to the strands running into the
undercrossing.

\begin{figure}
  \begin{center}
    \input{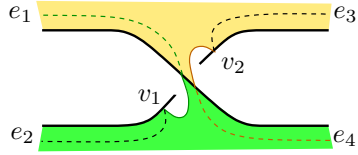}
  \end{center}
  \caption{Shown are portions of four ideal edges, terminating at
    undercrossings on a single crossing.  Ideal edges $e_1$ and $e_2$
    bound the same white disk and terminate at the ideal vertex $v_1$.
    Ideal edges $e_3$ and $e_4$ bound the same white disk and
    terminate at the ideal vertex $v_2$. Figure first appeared in
    \cite{fkp:gutsjp}.  }
  \label{fig:under-crossing}
\end{figure}

The surface $S_A$ runs through the crossing in a twisted rectangle.
Looking again at Figure \ref{fig:under-crossing}, note that the gold
portion of $S_A$ at the top of that figure is not cut off at the
crossing by an ideal edge terminating at $v_2$.  Instead, it follows $e_4$
through the crossing and along the underside of the figure, between
$e_4$ and the link.  Similarly for the green portion of $S_A$:
it follows the edge $e_1$ through the crossing and continues between
the edge and the link.

We visualize these thin portions of shaded face between an edge and
the link as \emph{tentacles}, and sketch them onto $H_A$ running from
the top--right of a segment to the base of the segment, and then along
an edge of $H_A$ (Figure
\ref{fig:tentacle}).  Similarly for the bottom--left.  

\begin{figure}
\includegraphics{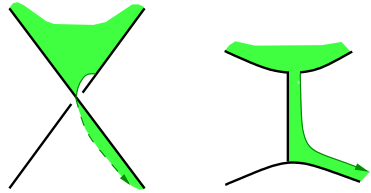}
\caption{Left: A tentacle\index{tentacle} continues a shaded face in
  the upper $3$--ball. Right: visualization of the tentacle on the
  graph $H_A$.}
\label{fig:tentacle}
\end{figure}

In Figure \ref{fig:tentacle}, we have drawn the tentacle by
removing a bit of edge of $H_A$.  When we do this for each segment,
top--right and bottom--left, the remaining connected components of
$H_A$ correspond exactly to ideal vertices.

Each ideal vertex of the upper polyhedron begins at the top--right (bottom--left) of
a segment, and continues along a state circle of $s_A$ until it meets another
segment.  In the diagram, this will correspond to an undercrossing.  As
we observed above, edges terminate at undercrossings.  Figure
\ref{fig:undercr-tentacles} shows the portions of ideal edges sketched
schematically onto $H_A$. 

\begin{figure}
  \input{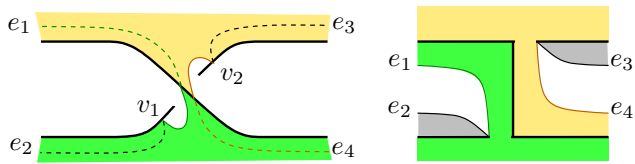}
  \caption{The ideal edges and shaded faces around the crossing of
    Figure \ref{fig:under-crossing}}
  \label{fig:undercr-tentacles}
\end{figure}

As for the shaded faces, we have seen that they extend in tentacles
through segments.  Where do they begin?  In fact, each shaded face
originates from an innermost disk.

To complete our combinatorial description of the upper
polyhedron, we color each innermost disk a unique color.  Starting
with the segments leading out of innermost disks, we sketch in
tentacles, removing portions of $H_A$.  Note that a tentacle will
continue past segments of $H_A$ on the opposite side of the state
circle without terminating, spawning new tentacles, but will terminate
at a segment on the same side of the state surface.  This is shown in
Figure \ref{fig:tentacle-mult}.  Since $H_A$ is a finite graph, the
process terminates.

\begin{figure}
  \begin{center}
    \input{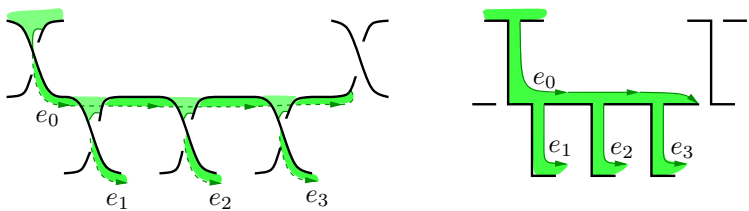}
  \end{center}
  \caption{Left: part of a shaded face in an upper $3$--ball. Right:
    the corresponding picture, superimposed on $H_A$.  The tentacle
    next to the ideal edge $e_0$ terminates at a segment on the same
    side of the state circle on $H_A$.  It runs past segments on the
    opposite side of the state circle, spawning new tentacles
    associated to ideal edges $e_1$, $e_2$, $e_3$.}
  \label{fig:tentacle-mult}
\end{figure}

An example is shown in Figure \ref{fig:tentacleupper}.  For this
example, there are two innermost disks, which we color green and gold.  Corresponding tentacles are shown.

\begin{figure}
  \begin{center}
    \includegraphics{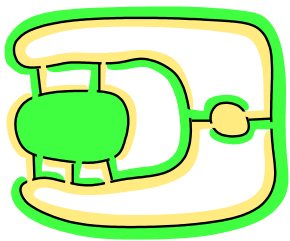}
  \end{center}
  \caption{An example of the combinatorics of the upper polyhedron. In this example, the polyhedron is combinatorially a prism over an ideal heptagon. This prism is an $I$--bundle over a subsurface of $S_A$.}
  \label{fig:tentacleupper}
\end{figure}

\subsection{Prime polyhedra detect fibers} 

In the last subsection, we saw that the
lower polyhedra correspond to alternating link diagrams.  It may happen that
one of these alternating diagrams is not prime. In other words, the alternating link diagram contains a pair of regions that meet along more than one edge. The polyhedral analogue of this notion is the following.

\begin{define}\label{def:prime-poly}
A combinatorial polyhedron $P$ is called \emph{prime} if every pair of faces of $P$ meet along at most one edge.
\end{define}

Primeness is an extremely desirable property in a polyhedral decomposition, for the following reason. The theory of \emph{normal surfaces}, developed by Haken \cite{haken:normal}, states that given a polyhedral decomposition of a manifold $M$, every essential surface $S \subset M$ can be moved into a form where it intersects each polyhedron in standard disks, called \emph{normal disks}. If this essential surface is (say) a compression disk $D \subset M_A$, then we may intersect $D$ with the union of white faces $\CalD$ to form a collection of arcs in $D$. An outermost arc must cut off a bigon. But, by Definition \ref{def:prime-poly}, a prime polyhedron cannot contain a bigon. Thus, once we obtain a decomposition of $M_A$ into prime polyhedra, Theorem \ref{thm:incompress} will immediately follow.

In practice, the polyhedra described in the last subsection may sometimes fail to be prime. Whenever this occurs, we need to cut them into smaller, prime pieces. Before we describe this cutting process, we record another powerful feature of the polyhedral decomposition.

Notice that the polyhedron in Figure \ref{fig:tentacleupper} is combinatorially a prism over an ideal heptagon: here, the two shaded faces are the horizontal faces of the prism, while the seven white faces are lateral, vertical essential product disks. In other words, the entire polyhedron is an $I$--bundle over an ideal polygon. It turns out that under the hypothesis of primeness,  when the upper polyhedron has this product structure, then so do all of the lower polyhedra, and so does the manifold $M_A = S^3 \cut S_A$.  

\begin{prop}[Lemma 5.8 of \cite{fkp:gutsjp}]\label{prop:product-detect}
Suppose that in the polyhedral decomposition of $M_A$ corresponding to an $A$--adequate diagram, every ideal polyhedron is prime.
Then,  the following are equivalent:
\begin{enumerate}
\item Every white face is an ideal bigon, i.e.\ an essential product disk.
\item The upper polyhedron is a prism over an ideal polygon.
\item Every polyhedron is a prism over an ideal polygon.
\item Every  region in $R$ that corresponds to a lower polyhedron is bounded by two state circles, connected by edges of $\GA$ that are identified to a single edge of $\GRA$.
\item Every edge of $\GRA$ is separating, i.e.\ $\GRA$ is a tree.
\end{enumerate}
\end{prop}

\begin{remark} In the decomposition process of $M_A$ that we have
described so far, the polyhedra may not be prime. However, we will see  in the next subsection that primeness can always be achieved after additional cutting.
\end{remark}

The proof of equivalence of the conditions in Proposition \ref{prop:product-detect} is not hard. For example, if every white face is a bigon, then in every polyhedron, those bigons must be strung end to end, forming a cycle of lateral faces in a prism. Note as well that if every polyhedron is a prism, i.e. a product, then this product structure extends over the bigon faces to imply that $M_A \cong S_A \times I$. This immediately gives one implication in Theorem \ref{thm:fiber-tree}: if $\GRA$ is a tree, then $S_A$ is a fiber.

The converse implication (if $S_A$ is a fiber, then $\GRA$ is a tree) requires knowing that our polyhedra \emph{detect} the JSJ decomposition of $M_A$. In other words: if part (or all) of $M_A$ is an $I$--bundle, then this $I$--bundle structure must be visible in the individual polyhedra. This property also follows from primeness.

\subsection{Ensuring Primeness}\label{subsec:primeness}

We have seen that primeness is a desirable property of the polyhedral decomposition. Here, we describe a way to detect when a lower polyhedron is not prime, and a way to fix this situation.

\begin{define}
The graph $H_A$ is \emph{non-prime} if there exists a state circle $C$
of $s_A$ and an arc $\alpha$ with both endpoints on $C$, such that the
interior of $\alpha$ is disjoint from $H_A$, and $\alpha$ is not isotopic into $C$ in the complement of $H_A$. The arc
$\alpha$ is called a \emph{non-prime arc}.
\label{def:non-prime}
\end{define}

Each non-prime arc is lies  in a single white face of $\CalD$, while its shadow on the soup can below lies in a single shaded face. Thus every non-prime arc indicates that a lower polyhedron violates Definition \ref{def:prime-poly}. See Figure \ref{fig:nonprime-bigon}.

Whenever we find a non-prime arc, we modify the polyhedral
decomposition as follows: push the arc $\alpha$ down against the soup
can of the state circle $C$.  This divides the corresponding lower
polyhedron into two.  Figure \ref{fig:nonprime-bigon} shows a
3--dimensional view of this cutting process.

\begin{figure}
  \includegraphics{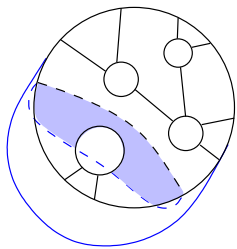}
  \caption{Split lower polyhedron along a non-prime arc.}
  \label{fig:nonprime-bigon}
\end{figure}

Combinatorially, we cut a lower polyhedron along the non-prime arc, as
in Figure \ref{fig:nonprime-lower}.  The lower polyhedra now
correspond to alternating links whose state circles contain $\alpha$.
On the boundary of the upper polyhedron, $\alpha$ meets two tentacles.
These will be joined into the same shaded face, by attaching both
tentacles to a regular neighborhood of $\alpha$.  See Figure
\ref{fig:nonprime-top}.

\begin{figure}
  \includegraphics{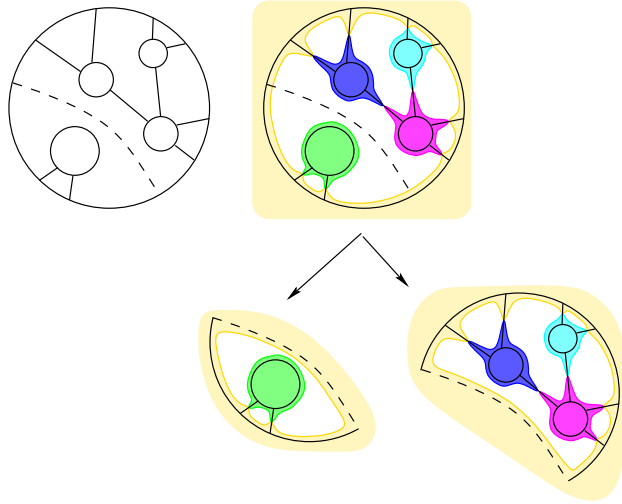}
  \caption{Splitting a lower polyhedron into two along a non-prime arc.}
  \label{fig:nonprime-lower}
\end{figure}

\begin{figure}
\includegraphics{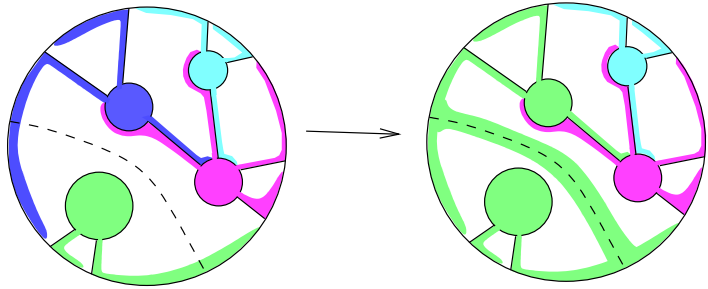}
\caption{Splitting the upper polyhedron along a non-prime arc.}
\label{fig:nonprime-top}
\end{figure}

Repeat this process of splitting along non-prime arcs until there are
no more non-prime arcs. This ensures that all the lower polyhedra are prime. Then, one can show that the upper polyhedron is prime as well.

 A decomposition along a maximal collection of
  non-prime arcs $\A$ is called a \emph{prime decomposition}. 
  Its main features are summarized as follows:
  \begin{itemize}
  \item It decomposes $M_A$ into one upper and at least one lower
    polyhedron.
    \item Every polyhedron is prime.
  \item Every polyhedron is checkerboard colored, with 4--valent vertices.
  \item White faces of the polyhedra correspond to regions of the
    complement of $H_A \cup \A$.
  \item Lower polyhedra are in one-to-one correspondence with
    nontrivial complementary regions of $s_A \cup
   \A$.
  \item Each lower polyhedron is identical to the
    checkerboard polyhedron of an alternating link, where the
    alternating link is obtained by taking the restriction of $H_A\cup \A$ to the corresponding region of $s_A
    \cup \A$, and replacing segments of $H_A$ with
    crossings (using the $A$--resolution).
\end{itemize}

Another important feature is that Proposition
\ref{prop:product-detect} holds for the prime polyhedral decomposition
that we have just described. It is worth noting that since the lower
polyhedra are are in one-to-one correspondence with nontrivial
complementary regions of $s_A \cup \A$, part (4) of the proposition
will refer to these regions of the diagram.

Having obtained a prime polyhedral decomposition, we may apply normal surface theory to study various pieces in the JSJ decomposition of
$M_A$. 
Recall that the JSJ decomposition yields three kinds of pieces:
$I$--bundles, solid tori, and the guts. To compute $\negeul(\guts(M_A))$, we need to understand the $I$--bundle components  that have nonzero Euler characteristic. We prove the following.

\begin{theorem}[Theorem 4.4 of \cite{fkp:gutsjp}]\label{thm:epd-span}
Let $B$ be an $I$--bundle component of the JSJ decomposition of $M_A$, such that $\chi(B) < 0$. Then $B$ is spanned by a collection of essential product
disks  $\{ D_1, \dots, D_n \}$, with the property that each $D_i$ is
embedded in a single polyhedron in the polyhedral decomposition of
$M_A$.
\end{theorem}

The EPDs
in the spanning set that lie in the \emph{lower} polyhedra of the
decompositions are well understood; they are in one-to-one
correspondence with 2--edge loops in the state graph $\GA$.  The EPDs
in the spanning set that lie in the upper polyhedron are
\emph{complex}; they are not obtainable in terms in the
lower polyhedra.  These are exactly the EPDs counted by the quantity $||E_c||$ of Theorems \ref{thm:guts-general}, \ref{thm:volume} and \ref{thm:no2loops}.
The EPDs in the upper polyhedron  also correspond to 2-edge loops in $\GA$, but the correspondence  is not one-to-one. In special cases of link diagrams, we can understand the combinatorial structure of the polyhedral decomposition well enough to show that $||E_c|| = 0$. This gives, for instance, Corollaries  \ref{cor:positive-vol-jones} and \ref{cor:jones-volumemonte}.

Our results about  normal surfaces in the polyhedral decomposition of $M_A$ can likely be used to attack other topological problems
about $A$--adequate links. We refer the reader  to  \cite[Chapter 10]{fkp:gutsjp}
for a detailed discussion of some of these open questions.

\bibliographystyle{hamsplain}
\bibliography{biblio}

\end{document}